\documentclass[english,a4paper,12pt]{amsart}
\usepackage[a4paper,lmargin=2cm,rmargin=2cm,tmargin=4cm,bmargin=4cm]{geometry}
\usepackage{amsmath,amssymb,amsfonts,amsthm}
\usepackage{amsrefs}
\usepackage{enumerate}
\usepackage{graphicx}
\usepackage{float}
\usepackage[colorlinks=true,linkcolor=blue]{hyperref}
\usepackage{pict2e}

\usepackage{bbm}
\usepackage[all]{xy}

\newtheorem{theorem}{Theorem}[section]
\theoremstyle{definition}
\newtheorem{rem}[theorem]{Remark}
\newtheorem{exem}[theorem]{Examples}
\newtheorem{defi}[theorem]{Definition}
\newtheorem{qu}[theorem]{Question}
\newtheorem{th1}[theorem]{Theorem}
\theoremstyle{plain}
\newtheorem{lem}[theorem]{Lemma}
\newtheorem{thm}[theorem]{Theorem}
\newtheorem{prop}[theorem]{Proposition}
\newtheorem{cor}[theorem]{Corollary}

\newcommand{\R}{\mathbb{R}}
\newcommand{\N}{\mathbb{N}}
\newcommand{\Z}{\mathbb{Z}}

\newcommand{\E}{\mathbb{E}}

\newcommand{\ra}{\sqrt}
\newcommand{\ep}{\varepsilon}
\newcommand{\disp}{\displaystyle}
\newcommand{\F}{\mathcal{F}}
\renewcommand{\l}{\ell}
\newcommand{\la}{\langle}
\renewcommand{\ra}{\rangle}

\newcommand{\vertiii}[1]{{\left\vert\kern-0.25ex\left\vert\kern-0.25ex\left\vert #1 
    \right\vert\kern-0.25ex\right\vert\kern-0.25ex\right\vert}}

\begin{document}

\title{Lipschitz-free spaces and Schur properties}

\author{Colin Petitjean}\address{Univ. Bourgogne Franche-Comt\'e\\
Laboratoire de Math\'ematiques UMR 6623\\
16 route de Gray\\
25030 Besan\c con Cedex\\
France}\email{colin.petitjean@univ-fcomte.fr}



\begin{abstract}
In this paper we study $\ell_1$-like properties for some Lipschitz-free spaces. The main result states that, under some natural conditions, the Lipschitz-free space over a proper metric space linearly embeds into an $\l_1$-sum of finite dimensional subspaces of itself. We also give a sufficient condition for a Lipschitz-free space to have the Schur property, the $1$-Schur property and the $1$-strong Schur property respectively. We finish by studying those properties on a new family of examples, namely the Lipschitz-free spaces over metric spaces originating from $p$-Banach spaces, for $p$ in $(0,1)$.
\end{abstract}

\maketitle

\section{Introduction}

Let $M$ be a pointed metric space and fix $0 \in M$ a distinguished origin of $M$. Let us denote $Lip_0(M)$ the space of all real valued Lipschitz functions $f$ defined on $M$ and verifying $f(0)=0$. This space equipped with the Lipschitz norm $\|\cdot\|_L$ (best Lipschitz constant) is a Banach space. For $x\in M$, define the evaluation functional $\delta_M(x) \in Lip_0(M)^*$ by $\langle\delta_M(x),f\rangle = f(x)$ for every $f \in  Lip_0(M)$. And now we let the Lipschitz-free space over M, denoted $\F(M)$, be the norm closed linear span in $Lip_0(M)^*$ of $\{\delta_M(x)  :  x\in M \}$. The map $\delta_M\colon x \in M \mapsto \delta_M(x)\in \F(M)$ is readily seen to be an isometry. Moreover the space $\F(M)$ enjoys a nice factorization property. For any Banach space $X$,  for any Lipschitz function $f\colon M \to X$, there exists a unique linear operator $\overline{f}\colon \F(M) \to X$ with $\|\overline{f}\|=\|f\|_{L}$ and such that the following diagram commutes
$$\xymatrix{
    M \ar[r]^f \ar@{^{(}->}[d]_{\delta_{M}} & X  \\
    \F(M) \ar[ru]_{\overline{f}} 
  }$$
Thus the map $f\in Lip_0(M,X) \mapsto \overline{f} \in \mathcal{L}(\F(M),X)$ is an onto linear isometry. As a direct consequence, $\F(M)$ is an isometric predual of $Lip_0(M)$. 

After the seminal paper \cite{godkal} of Godefroy and Kalton, Lispchitz free spaces have become an object of interest for many authors (see for instance \cites{cuthdoucha,dalprokauf,godozawa,hajekpernecka,hajeklancienpernecka,godard, kalton2004}). Indeed, the fundamental factorization property cited above transforms in a particular way a nonlinear problem into a linear one. This creates links between some old open problems in the geometry of Banach spaces and some open problems about Lipschitz-free spaces (see the open problems section in \cite{godsurvey}).

In this note we focus mostly on $\ell_1$-like properties like the Schur property or some stronger properties. In \cite{kalton2004}, Kalton proved that if $(M,d)$ is a metric space and $\omega$ is a nontrivial gauge then the space $\F(M,\omega \circ d)$ has the Schur property. In \cite{hajeklancienpernecka} H\'ajek, Lancien and Perneck\'a proved that the Lipschitz-free space over a proper countable metric space has the Schur property. Here we give a condition on $M$ which ensures that $\F(M)$ has the Schur property, and unifies the two above mentioned results. Furthermore, assuming that $M$ is proper (every closed ball is compact) and $\F(M)$ has the metric approximation property, we are able to provide more information about the``$\ell_1$-structure" of $\F(M)$. More precisely, we have the following theorem which is our first main result. 
\begin{thm} 
Let $M$ be a proper metric space such that $S_0(M)$ separates points uniformly (see Definition \ref{deflip}) and such that $\F(M)$ has (MAP). Then for any $\ep > 0$, there exists a sequence $(E_n)_n$ of finite-dimensional subspaces of $\F(M)$ such that $\F(M)$ is $(1+\ep)$-isomorphic to a subspace of $\left( \sum\disp{\oplus_n E_n } \right)_{\ell_1}$.
\end{thm}
We now describe the content of this paper. In Section \ref{section2}, we start by introducing some useful tools such as the little Lipschitz space over a metric space $M$ (denoted $lip_0(M)$) and the uniform separation of points. Next, generalizing the proof of Theorem 4.6 in \cite{kalton2004}, we show that $\F(M)$ has the Schur property whenever $lip_0(M)$ is $1$-norming. Next we move to the proof of Theorem \ref{propo} and we also show that it is optimal in some sense. Then we show that some quantitative versions of the Schur property are inherited by Lipschitz-free spaces satisfying some of the assumptions of Theorem \ref{propo}. 

In Section \ref{section3}, we introduce a new family of metric spaces where some results of Section \ref{section2} apply. This family consists of some metric spaces originating from $p$-Banach spaces ($p \in (0,1)$). We first focus on the metric space $M_p^n$ originating from $\ell_p^n$ for which we show that $\F(M_p^n)$ satisfies the hypothesis of Theorem \ref{propo}. Then we extend this result for arbitrary finite dimensional $p$-Banach spaces (Corollary \ref{cor37}). As a consequence we obtain our second main result.
\begin{thm}
Let $p$ be in $(0,1)$ and let $(X,\|\cdot\|)$ be a $p$-Banach space which admits a finite dimensional decomposition. Then $\F(X,\|\cdot\|^p)$ has the Schur property.
\end{thm}
Finally in Section \ref{section4} we give some open problems and make some related comments.

\section{Schur properties and Lipschitz-free spaces} \label{section2}

In this note, we only consider real Banach spaces and we write $X \equiv Y$ to say that the Banach spaces $X$ and $Y$ are linearly isometric.

\subsection{The little Lipschitz space and the uniform separation of points}
\begin{defi} \label{deflip}
Let $M$ be a metric space. We define the two following closed subspaces of $Lip_0(M)$ (the second one differs from the first one when $M$ is unbounded).
\begin{align*}
lip_0(M)&:=\left\{f\in Lip_0(M):  \lim\limits_{\varepsilon\rightarrow 0} \sup\limits_{0<d(x,y)<\varepsilon} \frac{|f(x)-f(y)|}{d(x,y)}=0\right\},\\
S_0(M)&:=\left\{ f\in lip_0(M) :  \lim\limits_{r\rightarrow \infty} \sup_{\stackrel{x \text{ or } y\notin B(0,r)}{x\neq y}} \frac{| f(x)-f(y)|}{d(x,y)}=0\right\}.
\end{align*}
\end{defi}

The first space $lip_0(M)$ is called the little Lipschitz space over $M$. This terminology first appeared in \cite{weaver} where they were considered on compact metric spaces.
\begin{defi}
We will say that a subspace $S\subset Lip_0(M)$ \emph{separates points uniformly} if there exists a constant $c\geq 1$ such that for every $x, y\in M$ there is $f\in S$ satisfying $||f||_L\leq c$ and $f(x)-f(y)=d(x,y)$. We then say that a subspace $Z$ of $X^*$, where $X$ is a Banach space, is \textit{$C$-norming} (with $C \geq 1$) if for every $x$ in $X$, $$\|x\| \leq C \disp{\sup_{z^* \in B_Z}|z^*(x)|}.$$ It is well known that $lip_0(M)$ or $S_0(M)$ are $C$-norming for some $C\geq 1$ if and only if they separate points uniformly (see Proposition 3.4 in \cite{kalton2004}).
\end{defi} 
Weaver showed in \cite{weaver} (Theorem 3.3.3) that if $M$ is a compact metric space then $lip_0(M)$ separates points uniformly if and only if it is an isometric predual of $\mathcal F(M)$. More generally, Dalet showed in \cite{daletproper} that a similar result holds for proper metric spaces. We state here this result for future reference.
\begin{thm}[Dalet] \label{propdalet}
Let $M$ be a proper metric space. Then $S_0(M)$ separates points uniformly if and only if it is an isometric predual of $\mathcal F(M)$.
\end{thm}

The following proposition will be useful in Section \ref{section3}. It provides conditions on $M$ to ensure that $lip_0(M)$ $1$-norming.

\begin{prop} \label{lemma1}
Let $M$ be a metric space. Assume that for every $x\neq y \in M$ and $\varepsilon>0$, there exist $N \subseteq M$ and a $(1+\ep)$-Lipschitz map $T$: $M \to N$ such that $lip_0(N)$ is $1$-norming for $\F(N)$, $d(Tx,x)\leq \varepsilon$ and $d(Ty,y) \leq \varepsilon$. Then $lip_0(M)$ is $1$-norming.
\end{prop}

\begin{proof}
Let $x\neq y \in M$ and $\varepsilon>0$. By our assumptions there exist $N \subseteq M$ and a $(1+\ep)$-Lipschitz map $T$: $M \to N$ such that $lip_0(N)$ is $1$-norming, $d(Tx,x)\leq \varepsilon$ and $d(Ty,y) \leq \varepsilon$. Since $lip_0(N)$ is $1$-norming there exists $f\in lip_0(N)$ verifying $\|f\|_L \leq 1+\ep $ and $|f(Tx)-f(Ty)|= d(Tx,Ty)$. Now we define $g=f \circ T$ on $M$. By composition $g$ is $(1+\ep)^2$-Lipschitz and  $g \in lip_0(M)$. Then a direct computation shows that $g$ does the work.
\begin{eqnarray*}
|g(x)-g(y)| &=& |f(Tx)-f(Ty)|= d(Tx,Ty) \\
&\geq& d(x,y) - d(x,Tx) - d(y,Ty) \\
&\geq& d(x,y) -2\ep.
\end{eqnarray*}
This ends the proof.
\end{proof}

\subsection{The Schur property}

We now turn to the study of the classical Schur property. We first recall its definition.

\begin{defi}
Let $X$ be a Banach space. We say that $X$ has \textit{the Schur property} if every weakly null sequence $(x_n)_{n \in \N}$ in $X$ is also $\| \cdot \|$-convergent to $0$. 
\end{defi}

According to \cite{kalton2004}, by a \textit{gauge} we mean a continuous, subadditive and increasing function $\omega:\mathbb [0,\infty)\longrightarrow [0,\infty)$ verifying $\omega(0)=0$ and $\omega(t)\geq t$ for every $t\in [0,1]$. We say that a gauge $\omega$ is \textit{non-trivial} whenever $\lim\limits_{t\rightarrow 0}\frac{\omega(t)}{t}=\infty$. In \cite{kalton2004} (Theorem 4.6), Kalton proved that if $(M,d)$ is a metric space and $\omega$ is a nontrivial gauge then the space $\F(M,\omega \circ d)$ has the Schur property.  A careful reading of his proof reveals that the key ingredient is actually the fact that $lip_0(M,\omega \circ d)$ is always $1$-norming (Proposition 3.5 in \cite{kalton2004}). This leads us to the following result. We include it here for completeness even though the proof is very similar to Kalton's original argument.

\begin{prop} \label{proposchur}
Let $M$ be a metric space such that $lip_0(M)$ is $1$-norming. Then the space $\F(M)$ has the Schur property.
\end{prop}

\begin{proof}
According to Proposition 4.3 in \cite{kalton2004}, for every $\ep >0$, $\F(M)$ is $(1+\ep)$-isomorphic to a subspace of $( \sum_{n \in \Z} \F(M_k) )_{\ell_1}$ where $M_k$ denotes the ball $\overline{B}(0,2^k) \subseteq M$ centered at $0$ and of radius $2^k$. Moreover the Schur property is stable under $\ell_1$-sums, under isomorphism and passing to subspaces. So it suffices to prove the result under the assumption that $M$ has finite radius. 

Let us consider $(\gamma_n)_n$ a normalized weakly null sequence in $\F(M)$. We will show that 
\begin{eqnarray} \label{lkk}
\forall \gamma \in \F(M), \, \liminf_{n \to + \infty} \|\gamma + \gamma_n\| \geq \|\gamma \| + \frac{1}{2},
\end{eqnarray}
from which it is easy to deduce that for every $\ep>0$, $(\gamma_n)_n$ admits a subsequence $(2 +\ep)$-equivalent to the $\ell_1$-basis (see the end of the proof of Proposition $4.6$ in \cite{kalton2004}). This contradicts the fact that $(\gamma_n)_n$ is weakly null.

Fix $\ep >0$ and $\gamma\in \F(M)$. We can assume that $\gamma$ is of finite support. Pick $f \in lip_0(M)$ with $\|f\|_L=1$ and $\la f, \gamma\ra > \|\gamma \| - \ep$. Next pick $\Theta >0 $ so that if $d(x,y) \leq \Theta$ then $|f(x) - f(y)| < \ep d(x,y)$. Choose $\delta < \frac{\ep \Theta}{2(1+\ep)} $. Then by Lemma 4.5 in \cite{kalton2004} we have
$$\inf_{|E|<\infty} \sup_{n} dist(\gamma_n, \F([E]_{\delta})) = 0,$$
where $[E]_{\delta} = \{x \in M : d(x,E)\leq \delta \}$. Thus there exists a finite set $E\subset M$ such that $E$ contains the support of $\gamma$ and such that for each $n$ we can find $\mu_n \in \F([E]_{\delta})$ with $\|\gamma_n - \mu_n\| < \ep$. Remark that $\F(E)$ is a finite dimensional space. Thus $\liminf_{n \to + \infty} dist(\gamma_n,\F(E)) \geq \frac{1}{2}$. Then by Hahn-Banach theorem, for every $n$ there exists $f_n \in Lip_0(M)$ verifying $\|f_n\|_{L}\leq 1+ \ep$, $f_n(E)=\{0\}$ and $\liminf_{n \to + \infty} \la f_n , \gamma_n \ra \geq \frac{1}{2}$. Now we define $g_n = (f+f_n)_{|[E]_{\delta}}$, then $g_n \in Lip_0([E]_{\delta})$ and we will show that $\|g_n\|_{L} < 1 +\ep$. We will distinguish two cases to show this last property. First suppose that $x$ and $y$ are such that $d(x,y)\leq \Theta$, then
\begin{eqnarray*}
|g_n(x)-g_n(y)| \leq |f(x) - f(y)| + |f_n(x)-f_n(y)| \leq \ep d(x,y) + d(x,y) = (1+\ep) d(x,y).
\end{eqnarray*}
Second if $x$ and $y$ are such that $d(x,y)> \Theta$, then there exist $u,v\in E$ with $d(x,u)\leq \delta$ and $d(y,v)\leq \delta$, so that
\begin{eqnarray*}
|g_n(x)-g_n(y)| &\leq&|f(x) - f(y)| + |f_n(x)| + |f_n(y)| \\
&=& |f(x) - f(y)| + |f_n(x) - f_n(u)| + |f_n(y) - f_n(v)| \\
&\leq& d(x,y) + 2(1+\ep)\delta \leq d(x,y) + \ep \Theta \leq (1+ \ep)d(x,y).
\end{eqnarray*}
We extend those functions $g_n$ to $M$ with the same Lipschitz constant and we still denote those extensions $g_n$ for convenience. We now estimate the desired quantities.
\begin{eqnarray*}
\|\gamma + \mu_n \| \geq \frac{1}{1+\ep} \la g_n, \gamma + \mu_n \ra = \frac{1}{1+\ep}(\la f, \gamma \ra + \la f, \mu_n \ra + \la f_n, \gamma \ra +\la f_n, \mu_n \ra),
\end{eqnarray*}
where
\begin{enumerate}[(i)]
\item $\la f, \gamma \ra > \|\gamma\| -\ep$.
\item $\limsup\limits_{n \to \infty} | \la f, \mu_n \ra | \leq \ep$, since $(\gamma_n)_n$ is weakly null and $\|\gamma_n-\mu_n\|<\ep$.
\item $\la f_n, \gamma \ra = 0$, since $\gamma \in \F(E)$.
\item $\liminf \limits_{n \to \infty} \la f_n, \mu_n \ra \geq \frac{1}{2} - \ep$, since  $\liminf \limits_{n \to  \infty} \la f_n, \gamma_n \ra > \frac{1}{2}$.
\end{enumerate}
Thus,
\begin{eqnarray*}
\liminf \limits_{n \to  \infty} \|\gamma + \gamma_n\| &\geq& \frac{1}{1+\ep}(\|\gamma \| + \frac{1}{2}- 3 \ep) - \ep. \\
\end{eqnarray*}
Since $\ep$ is arbitrary, this proves (\ref{lkk}).
\end{proof}

\subsection{Embeddings into \texorpdfstring{$\ell_1$}{l1}-sums}

Before stating the main result of this section we recall a few classical definitions. We say that a Banach space $X$ has the \textit{approximation property} (AP) if for every $\ep>0$, for every compact set $K \subset X$, there exists a finite rank operator $T \in \mathcal B(X)$ such that $\|Tx-x\| \leq \ep$ for every $x \in K$. Let $\lambda \geq 1$, if in the above definition T can always be chosen so that $\|T\|\leq \lambda$, then we say that $X$ has the \textit{$\lambda$-bounded approximation property } ($\lambda$-(BAP)). When $X$ has the $1$-(BAP) we say that $X$ has the \textit{metric approximation property} (MAP). We can now state and prove our first main result.

\begin{th1} \label{propo}
Let $M$ be a proper metric space such that $S_0(M)$ separates points uniformly and such that $\F(M)$ has (MAP). Then for any $\ep > 0$, there exists a sequence $(E_n)_n$ of finite-dimensional subspaces of $\F(M)$ such that $\F(M)$ is $(1+\ep)$-isomorphic to a subspace of $\left( \sum\disp{\oplus_n E_n } \right)_{\ell_1}$.
\end{th1}

\begin{proof}
The proof is based on three results. The first ingredient is the following lemma (Lemma $3.1$ in \cite{godkaltonli}):
\begin{lem}[Godefroy, Kalton and Li] \label{lemgodkalli}
Let $V$ be a subspace of $c_0(\N)$ with (MAP). Then for any $\ep > 0$, there exists a sequence $(E_n)_n$ of finite-dimensional subspaces of $V^*$ and a $weak^*$-to-$weak^*$ continuous linear map $T$: $V^* \to \left( \sum\disp{\oplus_n E_n } \right)_{\ell_1}$ such that for all $x^* \in V^*$
\begin{eqnarray*}
(1-\ep) \|x^*\| \leq \|Tx^*\| \leq (1+\ep) \|x^*\|. 
\end{eqnarray*}
\end{lem} 
Next, generalizing a proof of Kalton in the compact case (Theorem 6.6 in \cite{kalton2004}), Dalet has proved the following lemma (Lemma 3.9 in \cite{daletproper})
\begin{lem}[Dalet] \label{lemdalet} 
Let $M$ be a proper metric space. Then, for every $\ep>0$, the space $S_0(M)$ is $(1 + \ep)$-isomorphic to a subspace $Z$ of $c_0(\N)$.
\end{lem} 
Finally we need the following two results about (MAP) (see \cite{groth} and Theorem 1.e.15 in \cite{lintza1}).
\begin{thm}[Grothendieck] \label{grothendieck}
Let $X$ be a Banach space.
\begin{enumerate}[(G1).]
\item If $X^*$ has (MAP) then $X$ has (MAP).
\item If $X^*$ is separable and has (AP) then $X^*$ has (MAP). 
\end{enumerate}
\end{thm}
We are ready to prove Theorem \ref{propo}. Let us consider a metric space $M$ satisfying the assumptions of Theorem \ref{propo} and let us take $\ep >0$ arbitrary. Fix $\ep'$ such that $(1 + \ep')^3 < 1 +\ep $. According to Lemma \ref{lemdalet}, there exists a subspace $Z$ of $c_0(\N)$ such that $S_0(M)$ is $(1 + \ep')$-isomorphic to $Z$. Then note that $Z$ also has the metric approximation property. Indeed $Z^*$ is $(1 + \ep')$-isomorphic to $\F(M)$, so $Z^*$ has the $(1 + \ep')$-bounded approximation property. Next, using $(G2)$ of Theorem \ref{grothendieck} we get that $Z^*$ has MAP. Then using $(G1)$ of Theorem \ref{grothendieck} we get that $Z$ also has MAP. Thus we can apply Lemma \ref{lemgodkalli} to $Z$ so that there exists a sequence $(F_n)_n$ of finite-dimensional subspaces of $Z^*$ such that $Z^*$ is $(1+\ep')$-isomorphic to a subspace $F$ of $\left( \sum\disp{\oplus_n F_n } \right)_{\ell_1}$. Now $\F(M)$ is $(1 + \ep')$-isomorphic to $Z^*$ so there exists a sequence $(E_n)_n$ of finite-dimensional subspaces of $\F(M)$ such that $\left( \sum\disp{\oplus_n E_n } \right)_{\ell_1}$ is $(1+\ep')$-isomorphic to $\left( \sum\disp{\oplus_n F_n } \right)_{\ell_1}$. Then there exists a subspace $E$ of $\left( \sum\disp{\oplus_n E_n } \right)_{\ell_1}$ which is $(1+\ep')$-isomorphic to $F$. It is easy to check that $\F(M)$ is $(1 + \ep')^3$-isomorphic to $E$. This completes the proof.
\end{proof}
We now give some examples where Theorem \ref{propo} applies.

\begin{exem} \label{example}
The space $S_0(M)$ separates points uniformly and $\F(M)$ has (MAP) in any of the following cases.
\begin{enumerate}[1.]
\item $M$ proper countable metric space (Theorem 2.1 and Theorem 2.6 in \cite{daletproper}).
\item $M$ proper ultrametric metric space (Theorem 3.5 and Theorem 3.8 in \cite{daletproper}).
\item $M$ compact metric space such that there exists a sequence $(\ep_n )_n$
tending to $0$, a real number $\rho < 1/2$ and finite $\ep_n$-separated subsets $N_n$ of $M$ which are $\rho \ep_n$-dense in $M$ (Proposition 6 in \cite{godozawa}). For instance the middle-third Cantor set.
\end{enumerate}
\end{exem}
Before ending this section, we state a Proposition which says that Theorem \ref{propo} is optimal in some sense. Indeed, in our following example the dimension of the finite dimensional space $E_n$ in Theorem \ref{propo} has to go to infinity when $n$ goes to infinity.
\begin{prop}
There exists a countable compact metric space $K$, made of a convergent sequence and its limit, such that $\F(K)$ fails to have a cotype. In particular, $K$ satisfies the assumptions of Theorem \ref{propo} but does not embed isomorphically into $\ell_1$. 
\end{prop}
\begin{proof}
It is well known that $c_0$ has no nontrivial cotype. Since $c_0$ is separable, by Godefroy-Kalton lifting theorem (Theorem 3.1 in \cite{godkal}), there is a subspace of $\F(c_0)$ which is linearly isometric to $c_0$. Thus $\F(c_0)$ has no nontrivial cotype. So for every $n\geq1$, there exist $\gamma_{1}^n,\cdots,\gamma_{k_n}^n \in \F(c_0)$ such that 
$$\disp \left( \sum_{i=1}^{k_n} \|\gamma_{i}^n\|^n \right)^{\frac{1}{n}}> n \left(\E \| \sum_{i=1}^{k_n} \ep_i \gamma_{i}^n \|^n \right)^{\frac{1}{n}}, $$
where $(\ep_i)_{i=1}^{k_n}$ is an independent sequence of Rademacher random variables. Next we approximate each $\gamma_{i}^n$ by a finitely supported element $\mu_{i}^n \in\F(c_0)$ such that
$$\disp \left( \sum_{i=1}^{k_n} \|\mu_{i}^n\|^n \right)^{\frac{1}{n}}> \frac{n}{2} \left(\E \| \sum_{i=1}^{k_n} \ep_i \mu_{i}^n \|^n \right)^{\frac{1}{n}}. $$
Then we define $M_n = (\cup_{i=1}^{k_n}\, \mbox{supp}(\mu_{i}^n))\cup \{0\} \subset c_0$ which is a finite pointed metric space. Since scaling a metric space does not affect the linear isometric structure of the corresponding Lipschitz-free space, we may and do assume that the diameter of $M_n$ is less than $\frac{1}{2^{n+1}}$. 

Now we construct the desired compact pointed metric space as follows. Let us define the countable set $K:= (\cup_{n\geq 2} \{n\} \times M_n) \cup \{e\}$, $e$ being the distinguished point of $M$. To simplify the notation we write $M_n'$ for $\{n\} \times M_n$. Then we define a metric $d$ on $K$ such that $d(x,e)=\frac{1}{2^n}$ if $x\in M_n'$, $d(x,y)=d_{M_n}(x_n,y_n)$ if $x=(n,x_n), y=(n,y_n)\in M_n'$ and $d(x,y)=d(x,e)+d(y,e)$ if $x\in M_n'$, $y\in M_m'$ with $n \neq m$. Of course, with this metric $K$ is compact since it is a convergent sequence. Moreover the fact that $d(x,y)=d(x,e)+d(y,e)$ for $x\in M_n'$, $y\in M_m'$ with $n \neq m$ readily implies that $\F(K) \equiv (\sum \F(M_n \cup \{ e \}))_{\ell_1}$ (see Proposition 5.1 in \cite{kaufmann} for instance). By construction of $M_n$, $(\sum \F(M_n))_{\ell_1}$ has no nontrivial cotype. Thus $\F(K)$ also has no cotype. Therefore, $\F(K)$ cannot embed into $\ell_1$. But since it is a compact countable metric space, $lip_0(K)$ separates points uniformly and $\F(K)$ has (MAP) (Theorem 2.1 and 2.6 in \cite{daletproper}).
\end{proof}

\subsection{Quantitative versions of the Schur property} \label{section24}

It is well known that the Schur property is equivalent to the following condition: for every $\delta >0$, every $\delta$-separated sequence $(x_n)_{n \in \N}$ in the unit ball of $X$ contains a subsequence that is equivalent to the unit vector basis of $\ell_1$. That is there exists a constant $K>0$ (which may depend on the sequence considered) and a subsequence $(x_{n_k})_{k \in \N}$  such that
\begin{eqnarray*}
\sum_{i=1}^n |a_i| \geq \left \|\sum_{i=1}^n a_i x_{n_i} \right \| \geq \frac{1}{K} \sum_{i=1}^n |a_i| \, , \mbox{ for every }(a_i)_{i=1}^n \in \R^n.
\end{eqnarray*} 
In this case, we say that the sequence $(x_{n_k})_{k \in \N}$ is $K$-equivalent to the unit vector basis of $\ell_1$. This equivalence can be easily deduced using Rosenthal's $\ell_1$ theorem (see \cite{wnuk}). This last fact leads us to define the following quantitative version of the Schur property which has been introduced for the first time by Johnson and Odell in \cite{johnodell}. We also refer to \cite{godkaltonli} for the $1$-strong Schur property.

\begin{defi}
Let $X$ be a Banach space. We say that $X$ has \textit{the strong Schur property} if there exists a constant $K>0$ such that, for every $\delta >0$, any $\delta$-separated sequence $(x_n)_{n \in \N}$ in the unit ball of $X$ contains a subsequence that is $\frac{K}{\delta}$-equivalent to the unit vector basis of $\ell_1$. If in this definition, $K$ can be chosen so that $K=2+\ep$ for every $\ep >0$, then we say that $X$ has the \textit{$1$-strong Schur property}. 
\end{defi}
It is clear with the above characterization of the Schur property that the strong Schur property implies the Schur property. It is known that the Schur property is strictly weaker than the strong Schur property (see \cite{wnuk} or \cite{kalendaschur} for example). We refer the reader to \cite{kaltonc0extension} (Proposition $2.1$) for some equivalent formulations of the strong Schur property.
\begin{exem}
~
\begin{enumerate}[1.]
\item In \cite{knaustodell} (Proposition 4.1), Knaust and Odell proved that if $X$ has the property $(S)$ and does not contain any isomorphic copy of $\ell_1$, then $X^*$ has the strong Schur property. In particular $\ell_1$ and all its subspaces have the strong Schur property. A Banach space has the property $(S)$ if every normalized weakly null sequence contains a subsequence equivalent to the unit vector basis of $c_0$. This is known to be equivalent to the hereditary Dunford-Pettis property (Proposition 2 in \cite{cembranos})

\item In \cite{godkaltonli} (Lemma 3.4), Godefroy, Kalton and Li proved that a subspace of $L_1$ has the strong Schur property if and only if its unit ball is relatively compact in the topology of convergence in measure.
\end{enumerate}
\end{exem} 

We now give the second quantitative version of the Schur property which has been introduced  more recently by Kalenda and Spurn\'y in \cite{kalendaschur}.	

\begin{defi}
Let X be a Banach space, and let $(x_n)_{n \in \N}$ be a bounded sequence in $X$. We write $\mbox{clust}_{X^{**}}(x_n)$ for the set of all $weak^*$ cluster points of $(x_n)_{n \in \N}$ in $X^{**}$. Then we define the two following moduli:
\begin{eqnarray*}
\delta(x_n) &:=& \mbox{diam}\{ \mbox{clust}_{X^{**}}(x_n) \} \\
\mbox{ca}(x_n) &:=& \disp{\inf_{n \in \N}} \mbox{diam}\{ x_k \, ; \, k \geq n  \}.
\end{eqnarray*}
The first moduli measures how far is the sequence from being weakly Cauchy and the second one measures how far is the sequence from being $\|\cdot\|$-Cauchy. We then say that $X$ has \textit{the $C$-Schur property} if for every bounded sequence $(x_n)_{n \in \N}$ in $X$: $\mbox{ca}(x_n) \leq C \, \delta(x_n)$
\end{defi}

In \cite{kalendaschur} the authors proved that the $1$-Schur property implies the $1$-strong Schur property, and that the $1$-strong Schur property implies the $5$-Schur property. To the best of our knowledge, the question whether the $1$-strong Schur property implies the $1$-Schur property is open. 

In \cite{kalendaschurdunford} it is proved (Theorem $1.1$) that if $X$ is a subspace of $c_0(\Gamma)$, then $X^*$ has the $1$-Schur property. So we easily deduce the following proposition.

\begin{prop} \label{propo2}
Let $M$ be a proper metric space such that $S_0(M)$ separates points uniformly. Then $\F(M)$ has the $1$-Schur property.
\end{prop}

\begin{proof}
Fix $\ep >0$. We use Lemma \ref{lemdalet} to find a subspace $Z$ of $c_0(\N)$ which is $(1 + \ep)$-isomorphic to $S_0(M)$. Now, according to Theorem $1.1$ in \cite{kalendaschurdunford}, $Z^*$ has the $1$-Schur property. Since $S_0(M)$ separates points uniformly, we have that $S_0(M)^*\equiv\F(M)$ (Theorem \ref{propdalet}). Thus $\F(M)$ is $(1 + \ep)$-isomorphic to $Z^*$. Therefore $\F(M)$ has the $(1+\ep)$-Schur property. Since $\ep$ is arbitrary, $\F(M)$ has the $1$-Schur property.
\end{proof}

\begin{exem}
It is clear that metric spaces of the Examples \ref{example} satisfy the assumptions of Proposition \ref{propo2}. Moreover, the following family of examples also satisfies those last assumptions (see Proposition 2.5 in \cite{vectorvalued}). Let $M$ be a proper metric space and $\omega$ be a nontrivial gauge. Then $(M, \omega \circ d)$ is a proper metric space such that $S_0(M,\omega \circ d)$ separates points uniformly. Now this result leads us to wonder if $\F(M, \omega \circ d)$ has MAP for $M$ and $\omega$ as above (see Question \ref{qu2}).

\end{exem}

\section{Lipschitz-free spaces over metric spaces originating from p-Banach spaces}\label{section3}

In this section we study Lipschitz-free spaces over a new family of metric spaces, namely metric spaces originating from p-Banach spaces. We prove that results of the previous Section \ref{section2} can be applied to some spaces in this new family of examples. 

Let $X$ be a real vector space and $p$ in $(0,1)$. We say that a map $N$: $X \to [0,\infty)$ is $p$-subadditive if $N(x+y)^p \leq N(x)^p + N(y)^p$ for every $x$ and $y$ in $X$. Then a homogeneous and $p$-subadditive map $\|\cdot\|$: $X \to [0,\infty)$ is called a $p$-norm if $\|x\|=0$ if and only if $x=0$. Moreover the map $(x,y)\in X^2 \mapsto \|x-y\|^p$ defines a metric on $X$. If $X$ is complete for this metrizable topology, we say that $X$ is a $p$-Banach space. Note that a $p$-norm is actually a quasi-norm. That is there is a constant $C\geq 1$ such that for every $x,y \in X$: $\|x+y\| \leq C(\|x\|+\|y\|)$. Moreover an important theorem of Aoki and Rolewicz implies that every quasi-normed space can be renormed to be a $p$-normed space for some $p$ in $(0,1)$. For background on quasi-Banach spaces and $p$-Banach spaces we refer the reader to \cites{kaltonquasibanach,kaltonpeck}.

We fix $p$ in $(0,1)$ and consider $(X,\|\cdot\|)$ a $p$-Banach space. We denote $M_p=(X,d_p)$ the metric space where the metric is the $p$-norm of $X$ to the power $p$: $d_p(x , y )= \|x - y\|^p$ . Now $M_p$ being a metric space, we can study its Lipschitz-free space. At this point we would like to mention the paper \cite{albiackaltonquasi} of Albiac and Kalton in which they define and study Lipschitz-free spaces over $p$-Banach spaces (and not over metric spaces originating from $p$-Banach spaces, as we do here). They show among other things that the analogue of Godefroy-Kalton lifting theorem (Theorem 3.1 in \cite{godkal}) is false for $p$-Banach spaces.

As mentioned before, we know that $lip_0(M,\omega \circ d)$ is always $1$-norming when $\omega$ is a nontrivial gauge (Proposition 3.5 in \cite{kalton2004}). But in our case we consider a quasi-norm composed with the nontrivial gauge $\omega(t)=t^p$. Thus, we can expect to have the same result. We will see that we need to use more arguments to overpass the difference between a norm and a $p$-norm, that is the absence of the triangle inequality for the $p$-norm. However, techniques that are employed here are inspired by Kalton's ideas in \cite{kalton2004}.

As usual, $\|\cdot\|_1$ denotes the $\ell_1$-norm on $\R^n$. We also denote $\|\cdot\|_p$ the $p$-norm on $\R^n$ defined by $\|x\|_p= (\sum_{i=1}^n |x_i|^p)^{\frac{1}{p}}$ for every $x=(x_i)_{i=1}^n \in \R^n$. We begin with a very basic lemma. 
\begin{lem}\label{lemma2}
Let $p \in (0,1)$ and $n \in \mathbb N$. Then we have the following inequalities $$\forall x \in \mathbb R^n,\, \|x\|_1 \leq \|x\|_p \leq n^{\frac{1-p}{p}}\|x\|_1.$$ 
\end{lem}

\begin{proof}
The inequality $\|x\|_1 \leq \|x\|_p$ is obvious and a simple application of H{\"o}lder's inequality gives $\|x\|_p \leq n^{\frac{1-p}{p}}\|x\|_1$. 
\end{proof}

From now on we write $M_p^n$ for $(\R^n,d_p)=(\R^n,\|\cdot\|_p^p)$. In order to prove our first result about the structure of $\F(M_p^n)$, we need the following technical lemma.

\begin{lem} \label{lemphi}
Let $R\in (0,\infty)$, $p \in (0,1)$ and $n \in \N$. Then, there exists a Lipschitz function $\varphi$: $M_p^n \to M_p^n$ such that $\varphi$ is the identity map on $\overline{B}(0,R)$, is null on $M_p^n \backslash B(0,2R)$ and $\varphi$ is $n^{2-p}$-Lipschitz.
\end{lem}

\begin{proof}
Let us define $A=\overline{B}(0,R) \cup (M_p^n \backslash B(0,2R)) \subset M_p^n$ (balls are considered for $d_p$) and $\phi$: $(A,d_p) \to M_p^n$ such that $\phi$ is the identity on $\overline{B}(0,R)$ and is null on $M_p^n \backslash B(0,2R)$. It is easy to check that $\phi$ is $1$-Lipschitz. We now write $\phi=(\phi_1,\cdots,\phi_n)$. Then for every $k$ \linebreak $\phi_k$: $(A,d_p) \to (\R,|\cdot|^p)$ is $1$-Lipschitz. Thus $\phi_k$: $(A,\|\cdot\|_p) \to (\R,|\cdot|)$ is also $1$-Lipschitz (with the obvious extension of the notion of Lipschitz maps). Now the right hand side of the inequality in Lemma \ref{lemma2} implies that $\phi_k$: $(A,\|\cdot\|_1) \to (\R,|\cdot|)$ is  $n^{\frac{1-p}{p}}$-Lipschitz. So we can extend each $\phi_k$ without increasing the Lipschitz constant and we denote $\varphi_k$ those corresponding extensions. Summarizing we have $\varphi_k$: $(\R^n,\|\cdot\|_1) \to (\R,|\cdot|)$ which is $n^{\frac{1-p}{p}}$ Lipschitz and $\varphi_{k|A}=\phi_k$. Now the left hand side of the inequality in Lemma \ref{lemma2} implies that $\varphi_k$: $(\ell_p^n,\|\cdot\|_p) \to (\R,|\cdot|)$ is  $n^{\frac{1-p}{p}}$-Lipschitz. So $\varphi_k$: $M_p^n \to (\R,|\cdot|^p)$ is  $n^{1-p}$ Lipschitz. It follows easily that $\varphi=(\varphi_1,\cdots,\varphi_n)$: $M_p^n \to M_p^n$ is $n^{2-p}$-Lipschitz and verifies the desired properties.

\end{proof}
We are now able to prove the following result. 
\begin{prop} \label{prop2}
Let $p \in (0,1)$ and $n\in \N$. We recall our notation $(M_p^n,d_p)=(\R^n,\|\cdot\|_p^p)$. Then, the space $\F(M_p^n)$ is isometric to $S_0(M_p^n)^*$. In particular $S_0(M_p^n)$ is $1$-norming.
\end{prop}

\begin{proof}
In order to prove this result, we will first prove that $S_0(M_p^n)$ is $C_n$-norming for some $C_n>0$. And then we will deduce the desired result using Theorem \ref{propdalet}.

For every $m \in \N$ and $t\geq 0$, we define the following function $\omega_m(t)=\min \{ t^p , mt \}$ which is continuous, non-decreasing and subadditive. Note that $\lim\limits_{m \to + \infty}\omega_m(t) =t^p$.

Let $x\neq y \in M_p^n$. Since $(\ell_1^n)^* \equiv \ell_{\infty}^n $, by Hahn-Banach theorem there exists $x^* \in \ell_{\infty}^n$ such that $\|x^*\|_{\infty} = 1$ and $\la x^*,x-y \ra = \|x-y\|_1$. According to Lemma \ref{lemma2} this gives $\la x^*,x-y \ra \geq n^{\frac{p-1}{p}} \|x-y\|_p$. From now on we denote $F:=n^{\frac{1-p}{p}}x^*$ and we see $F$ as an element of $(l_p^n)^* \equiv l_{\infty}^n$ of norm $\|F\|_{(l_p^n)^*} \leq n^{\frac{1-p}{p}} $ which satisfies
\begin{equation}
|F(x)-F(y)|\geq \|x-y\|_p \label{IF}
\end{equation} 
Let us consider $R>2\max(\|x\|_p^p,\|y\|_p^p)$ and $\varphi: M_p^n \to M_p^n$ given by Lemma \ref{lemphi} (we denote $C$ its Lipschitz constant). Of course we can see $\varphi$ as a $C^{\frac{1}{p}}$-Lipschitz function from $\ell_p^n$ to $\ell_p^n$. We then consider $f_m$ defined on $M_p^n$ by

$$f_m(z)=\omega_m(|F(\varphi(z))-F(y)|)-\omega_m(|F(y)|).$$ 

Let us prove that those functions $f_m$ belong to $S_0(M_p^n)$ and do the job. For $z \neq z' \in M_p^n$ we compute
\begin{eqnarray*}
|f_m(z)-f_m(z')| &=& | \, \omega_m(|F(\varphi(z))-F(y)|) - \omega_m(|F(\varphi(z'))-F(y)|)\, |  \\
&\leq& \omega_m (|F(\varphi(z))- F(\varphi(z'))|)  \\
&=& \omega_m (|F(\varphi(z)-\varphi(z'))|). 
\end{eqnarray*}
By its definition $\omega_m(t)\leq t^p$. So we have 
$$|f_m(z)-f_m(z')| \leq |F(\varphi(z)-\varphi(z'))|^p \leq n^{1-p}d_p(\varphi(z),\varphi(z')) \leq C  n^{1-p} \, d_p(z,z') .$$
Thus, $f_m$ is $d_p$-Lipschitz with $\|f_m\|_{L} \leq C n^{1-p}$. Now since $\omega_m(t)\leq mt$ we get 
\begin{eqnarray*}
|f_m(z)-f_m(z')| &\leq&  m|F(\varphi(z)-\varphi(z'))| \leq  m n^{\frac{1-p}{p}} \|\varphi(z)-\varphi(z')\|_p \leq C^{\frac{1}{p}} m n^{\frac{1-p}{p}} \|z-z'\|_p \\
&\leq & (C^{\frac{1}{p}} m n^{\frac{1-p}{p}} \|z-z'\|_p^{1-p}) \, d_p(z,z').
\end{eqnarray*}
Since $1-p>0$, $\|z-z'\|_p^{1-p}$ and thus the Lipschitz constant of $f_m$ can be as small as we want for small distances. This provides the fact that $f_m \in lip_0(M_p^n)$. It remains to prove that $f_m$ satisfies the flatness condition at infinity to get $f_m \in S_0(M_p^n)$. To this end, fix $\ep >0$ and pick $k>2$ such that $\displaystyle{\frac{2C n^{1-p}}{(k-2)} \leq \ep}$. Now let $z$ and $z'$ be in $M$, and let us discuss by cases:
\begin{enumerate}[(i)]
\item If $z \not \in \overline{B}(0,kR)$ and $z' \not \in \overline{B}(0,2R)$, then $|f_m(z)-f_m(z')|=0 < \ep$.
\item If $z \not \in \overline{B}(0,kR)$ and $z' \in \overline{B}(0,2R)$, then
\begin{eqnarray*}
\frac{|f_m(z) - f_m(z')|}{d_{p}(z,z')} &\leq& \frac{|F(\varphi(z'))|^p}{(k-2)R} \\
&\leq& \frac{n^{1-p} \|\varphi(z')\|_p^p}{(k-2)R}   \\
&\leq & \frac{C n^{1-p} (2R)}{(k-2)R} \leq \ep.
\end{eqnarray*}
\end{enumerate}

Since $\ep$ is arbitrary, this proves that $f_m\in S_0(M)$. To finish the first part of the proof just notice now that using the inequality (\ref{IF}) and the fact that $\lim\limits_{m \to + \infty}\omega_m(t) =t^p$, we get $$|f_m(x)-f_m(y)| = \omega_m(|F(x)-F(y)|) \geq   \omega_m (\|x-y\|_p) \underset{m \to + \infty}{\longrightarrow} d_{p}(x,y).$$ Thus $S_0(M_p^n)$ is $Cn^{1-p}$-norming.

We are now moving to the duality argument. Remark that $M_p^n$ is a proper metric space, so using Theorem \ref{propdalet}, we have that $S_0(M_p^n)^*\equiv \F(M_p^n)$. Thus, obviously $S_0(M_p^n)$ is $1$-norming.
\end{proof}

Of course this last result still holds for every metric space originating from a $p$-Banach space $X_p$ of finite dimension.

\begin{cor} \label{cormp}
Let $p \in (0,1)$, $n\in \N$. Consider $(M_p,d_p)=(X,\|\cdot\|^p)$ where $(X,\|\cdot\|)$ is a $p$-Banach space of finite dimension. Then, the space $\F(M_p)$ is isometric to $S_0(M_p)^*$. In particular $S_0(M_p)$ is $1$-norming.
\end{cor}

\begin{proof}
Note that since $X_p$ is of finite dimension, it is isomorphic to $\ell_p^n$ for some $n\in \N$. Thus there is a bi-Lipschitz map between $M_p$ and $M_p^n$, let us say $L$: $M_p \to M_p^n$ is bi-Lipschitz with $C_1 d_{M_p}(x,y) \leq d_{M_p^n}(L(x),L(y)) \leq C_2 d_{M_p}(x,y)$. Now $S_0(M_p)$ is $\frac{C_2}{C_1}$-norming. Indeed pick $x \neq y \in M_p$ and $\ep >0$. Since $S_0(M_p^n)$ is $1$-norming there exists $f \in S_0(M_p^n)$ with Lipschitz constant less than $1+\ep$ such that $$|f(L(x))-f(L(y))| = d_{M_p^n}(L(x),L(y)) \geq C_1 d_{M_p}(x,y).$$ Now $f \circ L$: $M_p \to \R$ is Lipschitz with Lipschitz constant less than $C_2(1+\ep)$. Moreover as the composition of a bi-Lipschitz map with an element of $S_0(M_p^n)$ we know that $f \circ L \in S_0(M_p)$. Thus $S_0(M_p)$ is $\frac{C_2}{C_1}$-norming. Since $M_p$ is proper, it follows again from Theorem \ref{propdalet} that $\F(M_p)$ is isometric to $S_0(M_p)^*$ and thus $S_0(M_p)$ is $1$-norming.
\end{proof}

As announced, the assumptions of Theorem \ref{propo} are satisfied for a metric space originating from a $p$-Banach space of finite dimension.
\begin{cor} \label{cor37}
Let $p \in (0,1)$. Consider $(M_p,d_p)=(X,\|\cdot\|^p)$ where $(X,\|\cdot\|)$ is a $p$-Banach space of finite dimension. Then, there exists a sequence $(E_n)_n$ of finite-dimensional subspaces of $\F(M_p)$ such that $\F(M_p)$ is $(1+\ep)$-isomorphic to a subspace of $\left( \sum\disp{\oplus_n E_n } \right)_{\ell_1}$.
\end{cor}

\begin{proof}
The aim is to show that all assumptions of Theorem \ref{propo} are satisfied for $\F(M_p)$. According to Corollary \ref{cormp}, $S_0(M_p)$ is $1$-norming. Now it is proved in \cite{lancienpernecka} (Corollary 2.2) that if $M$ is a doubling metric space (that is there exists $D(M)\geq 1$ such that any ball $B(x,R)$ can be covered by $D(M)$ open balls of radius $R/2$) then $\F(M)$ has the BAP. Since $M_p\subset \R^n$, we get that $M_p$ is doubling. Thus in our case $\F(M_p)$ has the BAP. Since it is a dual space, we get from Theorem \ref{grothendieck} that $\F(M_p)$ actually has the (MAP). Thus all the assumptions of Theorem \ref{propo} are satisfied.
\end{proof}

We now turn to the study of the structure of $\F(M_p)$ with more general assumptions on $M_p$. In particular we now pass to infinite dimensional spaces and the aim is to explore the behavior of $\F(M_p)$ regarding properties such as the (MAP) and the Schur property. To do so, we will assume that $X$ is a $p$-Banach space which admits a Finite Dimensional Decomposition (shortened in FDD). In particular a space which admits a Schauder basis such as $\ell_p$ satisfies this assumption. We start with the study of the Schur property. Using our Proposition \ref{lemma1} we manage to prove the following result.

\begin{th1} \label{corlp}
Let $p$ in $(0,1)$ and let $(X,\|\cdot\|)$ be a $p$-Banach space which admits an FDD. Then $\F(X,\|\cdot\|^p)$ has the Schur property.
\end{th1}

\begin{proof}
First of all, note that we can assume that $X$ admits a monotone FDD. Indeed, it is classical that we can define an equivalent $p$-norm $\vertiii{\cdot}$ on $X$ such that the finite dimensional decomposition is monotone for $(X,\vertiii{\cdot})$ (see Theorem 1.8 in \cite{kaltonpeck} for instance). Now from the fact that $(X,\vertiii{\cdot})$ and $(X,\| \cdot \|)$ are isomorphic we deduce that $(X,\vertiii{\cdot}^p)$ and $(X,\| \cdot \|^p)$ are Lipschitz equivalent. Thus, by a routine argument for Lipschitz-free spaces, this implies that $\F(X,\vertiii{\cdot}^p)$ and 
$\F(X,\| \cdot \|^p)$ are isomorphic. Since the Schur property is stable under isomorphism, $\F(X,\vertiii{\cdot}^p)$ has the Schur property if and only if $\F(X,\| \cdot \|^p)$ has the Schur property. So from now on we assume that the FDD is monotone.

The aim is to apply Proposition \ref{lemma1}. We denote again $(M_p,d_p)=(X,\|\cdot\|^p)$. Since $X$ admits a monotone FDD, there exists a sequence $(X_k)_{k\in \N}$ of finite dimensional subspaces of $X$ such that every $x\in X$ admits a unique representation of the form $x=\sum_{k=1}^{\infty} x_k$ with $x_k \in X_k$. If we denote $P_n$ the projections from $X$ to $X_n$ defined by $P_n(x)=\sum_{k=1}^n x_k$ then $\sup_n \|P_n\|=1$. Notice that those projections are actually $1$-Lipschitz from $M_p$ to $M_{p,n}$ where $M_{p,n} = (X_n,d_p)$. 

Fix $x \neq y \in M_p$ and $\ep>0$. We can write $x=\sum_{k=1}^{\infty} x_k$, $y=\sum_{k=1}^{\infty} y_k$ with $x_k,y_k \in X_k$ for every $k$. Now fix $N \in \N$ such that $d_p(x,P_{N}(x))< \ep$ and $d_p(y,P_{N}(y))< \ep$. Since each $X_k$ is of finite dimension, the space $(\sum_{k=1}^N X_k,\|\cdot\|)$ is of finite dimension and thus by Corollary \ref{cormp}, $S_0(A)$ is $1$ norming where $A=(\sum_{k=1}^N M_{p,k},d_p)$. So in particular $lip_0(A)$ is $1$-norming. Thus, according to Proposition \ref{lemma1}, $lip_0(M_p)$ is $1$-norming and so $\F(M_p)$ has the Schur property by Proposition \ref{proposchur}.
\end{proof}

\begin{rem}
Notice that our Theorem \ref{corlp} is not a special case of Theorem 4.6 in \cite{kalton2004}. Indeed, in general we cannot write the distance $\|\cdot\|^p$ originating from a $p$-norm as the composition of a gauge and another distance. Let us prove this for instance in $M_p^2=(\R^2,\|\cdot\|_p^p)$. We argue by contradiction and so we assume that there exists $\omega$ a nontrivial gauge and $d$ a distance on $\R^2$ such that $\|x-y\|_p^p = \omega(d(x,y))$ for every $x$ and $y$ in $\R^2$. Now we consider the points $x=(t_x,0)$, $y=(0,t_y)$. Straightforward computations shows that $\|x-y\|_p^p = |t_x|^p + |t_y|^p = \|x\|_p^p + \|y\|_p^p$. Since $d$ is a distance and $\omega$ is a gauge we have 
\begin{eqnarray*}
|t_x|^p + |t_y|^p &=& \omega (d(x,y)) \leq \omega (d(x,0)+d(y,0)) \\
&\leq& \omega(d(x,0)) + \omega(d(y,0)) = |t_x|^p + |t_y|^p.
\end{eqnarray*}
Thus $\omega (d(x,0)+d(y,0)) = \omega(d(x,0)) + \omega(d(y,0))$. From where we deduce that $\omega$ is additive, and so is such that $\omega(t) = t\omega(1)=t$. This contradicts the fact that $\omega$ is a nontrivial gauge.

\end{rem}

We finish here by proving our last result about the (MAP). We keep the same notation as in Theorem \ref{corlp}.

\begin{prop}
Let $p \in (0,1)$ and $X$ be a $p$-Banach space which admits an FDD with decomposition constant $K$. Then $\F(X,\|\cdot\|^p)$ has the $K$-(BAP). In particular, if $X$ admits a monotne FDD then $\F(X,\|\cdot\|^p)$ has the (MAP).
\end{prop}

\begin{proof}
We still denote $(M_p,d_p)=(X,\|\cdot\|^p)$. Let $\mu_1,...,\mu_n \in \F(M_p)$ and $\ep>0$. Then there exists $N \in \N$ and $\nu_1,...,\nu_n \in \F(A)$ (where $A=(\sum_{k=1}^N M_{p,k},d_p)$) such that $\|\mu_k - \nu_k \| \leq \frac{\ep}{K}$. We have seen in the proof of Corollary \ref{cor37} that $\F(A)$ has (MAP). Thus, there exists $T$: $\F(A) \to \F(A)$ a finite rank operator such that $\|T\| \leq 1$ and $\|T\nu_k - \nu_k\| \leq \ep$ for every $k$. Since $P_N$: $M_p \to A$ is a $K$-Lipschitz retraction, the linearization $\hat{P_N}$: $\F(M_p) \to \F(A)$ is projection of norm at most $K$. This leads us to consider the operator $\hat{P_n}\circ T$: $\F(M_p) \to \F(M_p)$ for which direct computations show that it does the work. Indeed $\hat{P_n}\circ T$ is of finite rank, $\|\hat{P_n}\circ T\| \leq K$ and for every \nolinebreak$k$:
\begin{eqnarray*}
\|\hat{P_n}\circ T \mu_k - \mu_k\| &\leq& \|\hat{P_n}\circ T \mu_k - \hat{P_n}\circ T \nu_k\| + \|\hat{P_n}\circ T \nu_k - \nu_k\| + \|\mu_k - \nu_k\| \\
&\leq & \|\hat{P_n}\circ T\|  \|\mu_k - \nu_k\| + \|T \nu_k - \nu_k\| +   \ep\\
&\leq & 3 \ep .
\end{eqnarray*}
\end{proof}

\section{Final comments and open questions} \label{section4}

In Proposition \ref{proposchur} we stated that if $lip_0(M)$ is $1$-norming then the space $\F(M)$ has the Schur property. It is then natural to try to relax the assumption. For instance, we ask the following question
\begin{qu} \label{qu1}
Let $M$ be metric space such that $lip_0(M)$ is $C$-norming for some $C>1$ ? Does $\F(M)$ has the Schur property ?
\end{qu}

Now remark that we can deduce Proposition $\ref{propo2}$ as a direct consequence of Theorem $\ref{propo}$ under the additional assumption that $\F(M)$ has MAP. However, it seems that most of the examples of proper (in particular compact) metric spaces that we can find in the literature are as follows. Either $S_0(M)$ separates points uniformly and $\F(M)$ has (MAP), or $S_0(M)$ does not separate points uniformly and $\F(M)$ does not have (MAP). So we wonder:
\begin{qu} \label{qu2}
Let $M$ be proper metric space such that $S_0(M)$ separates points uniformly. Then does $\F(M)$ have (MAP) ? In particular, if $M$ is proper and if $\omega$ is a nontrivial gauge, does $\F(M,\omega \circ d)$ have (MAP) ?
\end{qu}
In Section \ref{section4}, we proved (Proposition \ref{corlp}) that Lipschitz-free spaces, over some metric spaces originating from $p$-Banach spaces, have the Schur property. It is then natural to wonder if we can extend this result to a larger class of metric spaces. Surprisingly, it is really easy to see that we cannot extend this result to every metric space originating from a $p$-Banach space. Indeed, consider the metric space $M$ originating from $L_p[0,1]$. Then the map $\varphi \colon t \in [0,1] \mapsto \mathbbm{1}_{[0,t]} \in M$ is a nonlinear isometry. Therefore there is a linear and isometric embedding of $\F([0,1])=L_1([0,1])$ into $\F(M)$. Consequently, since the Schur property is stable under passing to subspaces and $L_1$ does not have it, $\F(M)$ does not have the Schur property. Furthermore, using the same ideas one can show that $\F([0,1]^2)$ linearly embeds into $\F(M)$. Thus, using the fact that $\F(\R^2)$ does not embed into $L_1$ (see \cite{naorsche}), we get that $\F(M)$ does not embed into $L_1$. A major difference between $L_p$ and the $p$-Banach spaces studied in Section \ref{section3} is that $L_p$ has trivial dual. In particular, the dual does not separate points of $L_p$. This suggests the following question.
\begin{qu} \label{qu3}
Consider $(X,\|\cdot\|)$ a $p$-Banach space whose dual $X^*$ separates points. Then does $\F(X,\|\cdot\|^p)$ have the Schur property ?
\end{qu}

\textbf{Acknowledgments.} The author is grateful to G. Lancien and A. Proch\'azka for useful conversations. Finally, the author is grateful to the referee for useful comments and suggestions which permitted to improve the presentation of this paper.

\bibliographystyle{elsarticle}
\bibliography{mabiblio}

\end{document}